\documentclass[11pt]{amsart}
\usepackage{amssymb}
\usepackage{amsfonts}
\usepackage{amssymb,amsfonts}
\allowdisplaybreaks

\begin{document}
\theoremstyle{plain}
\newtheorem{thm}{Theorem}[section]
\newtheorem{theorem}[thm]{Theorem}
\newtheorem{lemma}[thm]{Lemma}
\newtheorem{corollary}[thm]{Corollary}
\newtheorem{corollary*}[thm]{Corollary*}
\newtheorem{proposition}[thm]{Proposition}
\newtheorem{proposition*}[thm]{Proposition*}
\newtheorem{conjecture}[thm]{Conjecture}
\theoremstyle{definition}
\newtheorem{construction}[thm]{Construction}
\newtheorem{notations}[thm]{Notations}
\newtheorem{question}[thm]{Question}
\newtheorem{problem}[thm]{Problem}
\newtheorem{remark}[thm]{Remark}
\newtheorem{remarks}[thm]{Remarks}
\newtheorem{definition}[thm]{Definition}
\newtheorem{claim}[thm]{Claim}
\newtheorem{assumption}[thm]{Assumption}
\newtheorem{assumptions}[thm]{Assumptions}
\newtheorem{properties}[thm]{Properties}
\newtheorem{example}[thm]{Example}
\newtheorem{comments}[thm]{Comments}
\newtheorem{blank}[thm]{}
\newtheorem{observation}[thm]{Observation}
\newtheorem{defn-thm}[thm]{Definition-Theorem}

\newcommand{\sM}{{\mathcal M}}


\title{A proof of the Faber intersection number conjecture}
        \author{Kefeng Liu}
        \address{Center of Mathematical Sciences, Zhejiang University, Hangzhou, Zhejiang 310027, China;
                Department of Mathematics,University of California at Los Angeles,
                Los Angeles, CA 90095-1555, USA}
        \email{liu@math.ucla.edu, liu@cms.zju.edu.cn}
        \author{Hao Xu}
        \address{Center of Mathematical Sciences, Zhejiang University, Hangzhou, Zhejiang 310027, China}
        \email{haoxu@cms.zju.edu.cn}

        \begin{abstract}
        We prove the famous Faber intersection number conjecture and other more general results
         by using a recursion formula of $n$-point functions
        for intersection numbers on moduli spaces of curves. We also
        present some vanishing properties of Gromov-Witten invariants.
        \end{abstract}
    \maketitle

\section{Introduction}
Starting from the work of Mumford, one fundamental problem in
algebraic geometry is the study of intersection theory on moduli
spaces of stable curves. Through the work of Witten and Kontsevich
we learned that the intersection theory of moduli spaces also has
striking connection to string theory and two dimensional gravity.
Denote by $\overline{\sM}_{g,n}$ the moduli space of stable
$n$-pointed genus $g$ complex algebraic curves. We have the morphism
that forgets the last marked point
$$
\pi: \overline{\sM}_{g,n+1}\longrightarrow \overline{\sM}_{g,n}.
$$
Denote by $\sigma_1,\dots,\sigma_n$ the canonical sections of $\pi$,
and by $D_1,\dots,D_n$ the corresponding divisors in
$\overline{\sM}_{g,n+1}$. Let $\omega_{\pi}$ be the relative
dualizing sheaf, we have the following tautological classes on
moduli spaces of curves.
\begin{align*}
\psi_i&=c_1(\sigma_i^*(\omega_{\pi}))\\
\kappa_i&=\pi_*\left(c_1\left(\omega_{\pi}\left(\sum D_i\right)\right)^{i+1}\right)\\
\lambda_k&=c_k(\mathbb E),\quad 1\leq k\leq g,
\end{align*}
where $\mathbb E=\pi_*(\omega_{\pi})$ is the Hodge bundle.

Intuitively, $\psi_i$ is the first Chern class of the line bundle
corresponding to the cotangent space of the universal curve at the
$i$-th marked point and the fiber of $\mathbb E$ is the space of
holomorphic one forms on the algebraic curve.

 The classes $\kappa_i$ were
first introduced by Mumford \cite{Mu} on $\overline{\sM}_g$, their
generalization to $\overline{\sM}_{g,n}$ here is due to
Arbarello-Cornalba \cite{ArCo}.

We use Witten's notation
$$\langle\tau_{d_1}\cdots\tau_{d_n}\kappa_{a_1}\cdots\kappa_{a_m}\mid\lambda_{1}^{k_{1}}
    \cdots\lambda_{g}^{k_{g}}\rangle\triangleq \int_{\overline{\mathcal{M}}_{g,n}}\psi_{1}^{d_{1}}\cdots\psi_{n}^{d_{n}}\kappa_{a_1}\cdots\kappa_{a_m}\lambda_{1}^{k_{1}}
    \cdots\lambda_{g}^{k_{g}}.$$
These intersection numbers are called the Hodge integrals. They are
rational numbers because the moduli space of curves are orbifolds
(with quotient singularities) except in genus zero. Their degrees
should add up to $\dim \overline{\mathcal{M}}_{g,n}= 3g-3+n$.

Intersection numbers of pure $\psi$ classes
$\langle\tau_{d_1}\cdots\tau_{d_n}\rangle$ are often called
intersection indices or descendant integrals. Faber's algorithm
\cite{Fa2} reduces the calculation of general Hodge integrals to
intersection indices, based on Mumford's Chern character formula
\cite{Mu}
$$
{\rm ch}_{2g-1}(\mathbb E)= \frac{ B_{2g}}{(2g)!}\left[
\kappa_{2g-1}-\sum_{i=1}^{n}\psi_i^{2g-1}+\frac12\sum_{\xi\in\Delta}{l_{\xi}}_*
\left(\sum_{i=0}^{2g-2}\psi_{n+1}^i(-\psi_{n+2})^{2g-2-i}\right)\right],
$$
where $\Delta$ enumerates all boundary divisors and ${l_{\xi}}_*$ is
the push-forward map under the natural inclusion.

 The
celebrated Witten-Kontsevich theorem \cite{Ko, Wi} asserts that the
generating function of intersection indices
\begin{equation*}\label{gen}
F(t_0, t_1, \ldots)= \sum_{g} \sum_{\bold n}
\langle\prod_{i=0}^\infty \tau_{i}^{n_i}\rangle_{g}
\prod_{i=0}^\infty \frac{t_i^{n_i} }{n_i!}
\end{equation*}
is governed by the KdV hierarchy, which provides a recursive way to
compute all these intersection numbers.

The tautological ring $\mathcal R^*(\sM_g)$ is defined to be the
smallest $\mathbb Q$-subalgebra of the Chow ring $\mathcal
A^*(\sM_g)$ generated by the tautological classes $\kappa_i$ and
$\lambda_i$. Mumford \cite{Mu} proved that the ring $\mathcal
R^*(\sM_g)$ is in fact generated by the $g-2$ classes
$\kappa_1,\dots, \kappa_{g-2}$.

It is a theorem of Looijenga \cite{Lo} that $\dim\mathcal
R^k(\mathcal M_g)=0,\ k>g-2$ and $\dim\mathcal R^{g-2}(\mathcal
M_g)\leq1$. Later Faber proved that actually $\dim\mathcal
R^{g-2}(\mathcal M_g)=1$.

\subsection*{Faber's conjecture}

Around 1993, Faber \cite{Fa} proposed three remarkable conjectures
about the structure of the tautological ring $\mathcal R^*(\sM_g)$
which we briefly state as follows:

\begin{enumerate}
\item[i)] For $0 \leq k \leq g-2$, the natural
product $$R^k(\mathcal M_g) \times R^{g-2-k}(\mathcal M_g)
\rightarrow R^{g-2}(\mathcal M_g) \cong \mathbb Q$$ is a perfect
pairing.

\item[ii)] The $[g/3]$ classes $\kappa_1,\dots,\kappa_{[g/3]}$
generate the ring $\mathcal R^*(\sM_g)$, with no relations in
degrees $\leq [g/3]$.

\item[iii)] Let $\sum_{j=1}^n d_j=g-2$ and $d_j\geq0$. Then
\begin{equation}\label{eqfa1}
\pi_*(\psi_1^{d_1+1}\dots\psi_n^{d_n+1})=\sum_{\sigma\in
S_n}\kappa_\sigma=\frac{(2g-3+n)!}{(2g-2)!!\prod_{j=1}^{n}(2d_j+1)!!}\kappa_{g-2},
\end{equation}
where $\kappa_\sigma$ is defined as follows: write the permutation
$\sigma$ as a product of $\nu(\sigma)$ disjoint cycles
$\sigma=\beta_1\cdots\beta_{\nu(\sigma)}$, where we think of the
symmetric group $S_n$ as acting on the $n$-tuple $(d_1,\dots ,d_n)$.
Denote by $|\beta|$ the sum of the elements of a cycle $\beta$. Then
$ \kappa_\sigma=\kappa_{|\beta_1|}\kappa_{|\beta_2|}\dots
\kappa_{|\beta_{\nu(\sigma)}|}$.

\end{enumerate}

Part (i) is called Faber's perfect pairing conjecture, which is
still open. Faber has verified it for $g\leq23$.

Part (ii) has been proved independently by Morita \cite{Mo} and
Ionel \cite{Io} with very different methods. As pointed out by Faber
\cite{Fa}, Harer's stability result implies that there is no
relation in degrees $\leq [g/3]$.

Part (iii) of Faber's conjectures is the intersection number
conjecture, whose importance lies in that it computes all top
intersections in the tautological ring $\mathcal R^*(\sM_g)$ and
determines its ring structure if we assume Faber's perfect pairing
conjecture. Theoretically it gives the dimension of tautological
rings by computing the rank of intersection matrices which we will
discuss in a subsequent work.

Faber's conjecture is a fundamental question mentioned in monographs
such as \cite{Gat, HaMo} that many algebraic geometers have worked
on. In this paper, we prove the Faber intersection number conjecture
completely. First we recall two equivalent formulations.

The Faber intersection number conjecture is equivalent to
\begin{equation}\label{eqfa2}
\int_{\overline{\sM}_{g,n}}\psi_1^{d_1}\dots\psi_n^{d_n}\lambda_g\lambda_{g-1}=
\frac{(2g-3+n)!|B_{2g}|}{2^{2g-1}(2g)!\prod_{j=1}^{n}(2d_j-1)!!},
\end{equation}
where $B_{2g}$ denotes the $2g$-th Bernoulli number. By Mumford's
formula for the Chern character of the Hodge bundle, the above
identity is equivalent to
\begin{align}\label{eqfa3}
\frac{(2g-3+n)!}{2^{2g-1}(2g-1)!\prod_{j=1}^n(2d_j-1)!!}
=&\langle\tau_{2g}\prod_{j=1}^n\tau_{d_j}\rangle_g-\sum_{j=1}^{n}
\langle\tau_{d_j+2g-1}\prod_{i\neq j}\tau_{d_i}\rangle_g\nonumber\\
&+\frac{1}{2}\sum_{j=0}^{2g-2}(-1)^j\langle\tau_{2g-2-j}\tau_j\prod_{i=1}^n\tau_{d_i}\rangle_{g-1}\\\nonumber
&+\frac{1}{2}\sum_{\underline{n}=I\coprod
J}\sum_{j=0}^{2g-2}(-1)^j\langle\tau_{j}\prod_{i\in
I}\tau_{d_i}\rangle_{g'}\langle\tau_{2g-2-j}\prod_{i\in
J}\tau_{d_i}\rangle_{g-g'},
\end{align}
where $d_j\geq1$, $\sum_{j=1}^{n}d_j=g+n-2$. We refer to \cite{Fa,
LX1} for discussions of the above equivalences.

The following interesting relation is observed by Faber and proved
by Zagier using the Faber intersection number conjecture (see
\cite{Fa})
$$\kappa_1^{g-2}=\frac{1}{g-1}2^{2g-5}((g-2)!)^2\kappa_{g-2}.$$
In fact, from \eqref{eqfa1}, the above relation is equivalent to a
combinatorial identity
\begin{equation*}
\sum_{k=1}^g\left(\frac{(-1)^k}{k!}(2g+1+k)\sum_{g=m_1+\dots+m_k\atop
m_i>0}\binom{2g+k}{2m_1+1,\dots,2m_k+1}\right)=(-1)^g2^{2g}(g!)^2.
\end{equation*}
We learned of an elegant proof from Jian Zhou using the residue
theorem.

Faber \cite{Fa} proved identity \eqref{eqfa3} when $n=1$ using
explicit formulae of up to three-point functions. The identity
\eqref{eqfa2} was shown to follow from the degree 0 Virasoro
conjecture for $\mathbb P^2$ by Getzler and Pandharipande
\cite{Ge-Pa}. In 2001 Givental \cite{Giv} has announced a proof of
Virasoro conjecture for $\mathbb P^n$. Y.-P. Lee and R.
Pandharipande are writing a book \cite{LP} giving details. Recently
Teleman \cite{Te} announced a proof of the Virasoro conjecture for
manifolds with semi-simple quantum cohomology. His argument depends
crucially on the Mumford conjecture about the stable rational
cohomology rings of the moduli spaces proved by Madsen and Weiss
\cite{MW}.

Goulden, Jackson and Vakil \cite{GJV} recently give an enlightening
proof of identity \eqref{eqfa1} for up to three points. Their
remarkable proof uses relative virtual localization and a
combinatorialization of the Hodge integrals, establishing
connections to double Hurwitz numbers.

Our alternative approach is quite direct, we prove identity
\eqref{eqfa3} for all $g$ and $n$ by using a recursive formula of
$n$-point functions. Actually, the $n$-point function formula has
far-reaching applications. Recently Zhou \cite{Zh} used our results
on $n$-point functions in his computation of Hurwitz-Hodge
integrals, which leads to a proof of the crepant resolution
conjecture of type A surface singularities for all genera.

\

\noindent{\bf Acknowledgements.} The authors would like to thank
Professors Sergei Lando, Jun Li, Chiu-Chu Melissa Liu, Ravi Vakil
and Jian Zhou for helpful communications.

\vskip 30pt
\section{The n-point functions}

\begin{definition}
We call the following generating function
$$F(x_1,\dots,x_n)=\sum_{g=0}^\infty F_g(x_1,\dots,x_n)=\sum_{g=0}^\infty\sum_{\sum d_j=3g-3+n}\langle\tau_{d_1}\cdots\tau_{d_n}\rangle_g\prod_{j=1}^n x_j^{d_j}$$
the $n$-point functions.
\end{definition}

Consider the following ``normalized'' $n$-point function
$$G(x_1,\dots,x_n)=\exp\left(\frac{-\sum_{j=1}^n
x_j^3}{24}\right) F(x_1,\dots,x_n).$$ We will let
$G_g(x_1,\dots,x_n)$ denote the degree $3g-3+n$ homogenous component
of $G(x_1,\dots,x_n)$.

In contrast with the original $n$-point function, its normalization
has some distinct properties (see \cite{LX2}). For example, the
coefficient of $z^k\prod_{j=1}^{n}x_j^{d_j}$ in
$G_g(z,x_1,\dots,x_n)$ is zero whenever $k>2g-2+n$.

It's well-known that
$$F_0(x_1,\dots,x_n)=G_0(x_1,\dots,x_n)=(x_1+\cdots+x_n)^{n-3}.$$

There are explicit formulae for one and two-point functions due to
Witten \cite{Wi} and Dijkgraaf (see \cite{Fa}) respectively
$$G(x)=\frac{1}{x^2},\qquad G(x,y)=\frac{1}{x+y}\sum_{k\geq0}\frac{k!}{(2k+1)!}\left(\frac{1}{2}xy(x+y)\right)^k.$$
In an unpublished note \cite{Za} (kindly sent to us by Faber),
Zagier obtained a marvelous formula of the three-point function (see
\cite{LX2}).

We proved in \cite{LX2} the following recursion formula for general
normalized $n$-point function.
\begin{proposition} \cite{LX2} For $n\geq2$,
\begin{equation*}
G(x_1,\dots,x_n)=\sum_{r,s\geq0}\frac{(2r+n-3)!!}{4^s(2r+2s+n-1)!!}P_r(x_1,\dots,x_n)\Delta(x_1,\dots,x_n)^s,
\end{equation*}
where $P_r$ and $\Delta$ are homogeneous symmetric polynomials
defined by
\begin{align*}
\Delta(x_1,\dots,x_n)&=\frac{(\sum_{j=1}^nx_j)^3-\sum_{j=1}^nx_j^3}{3},\nonumber\\
P_r(x_1,\dots,x_n)&=\left(\frac{1}{2\sum_{j=1}^nx_j}\sum_{\underline{n}=I\coprod
J}(\sum_{i\in I}x_i)^2(\sum_{i\in J}x_i)^2G(x_I)
G(x_J)\right)_{3r+n-3}\nonumber\\
&=\frac{1}{2\sum_{j=1}^nx_j}\sum_{\underline{n}=I\coprod
J}(\sum_{i\in I}x_i)^2(\sum_{i\in J}x_i)^2\sum_{r'=0}^r
G_{r'}(x_I)G_{r-r'}(x_J),
\end{align*}
where $I,J\ne\emptyset$, $\underline{n}=\{1,2,\ldots,n\}$ and
$G_g(x_I)$ denotes the degree $3g+|I|-3$ homogeneous component of
the normalized $|I|$-point function $G(x_{k_1},\dots,x_{k_{|I|}})$,
where $k_j\in I$.
\end{proposition}
The proof amounts to check that $G(x_1,\dots,x_n)$, as recursively
defined in Proposition 2.2, satisfies the following
Witten-Kontsevich differential equation (see \cite{LX2}),
\begin{multline*}
y\frac{\partial}{\partial
y}\left((y+\sum_{j=1}^{n}x_j)^2G_g(y,x_1,\dots,x_n)\right)\\
=\frac{y}{8}(y+\sum_{j=1}^{n}x_j)^{4}G_{g-1}(y,x_1,\dots,x_n)-\frac{y^3}{8}(y+\sum_{j=1}^{n}x_j)^2 G_{g-1}(y,x_1,\dots,x_n)\\
+\frac{y}{2}\sum_{\underline{n}=I\coprod J}\left(\left(y+\sum_{i\in
I}x_i\right)\left(\sum_{i\in J}x_i\right)^3+2\left(y+\sum_{i\in
I}x_i\right)^2\left(\sum_{i\in
J}x_i\right)^2\right)G_{g'}(y,x_I)G_{g-g'}(x_J)\\
-\frac{1}{2}\left(y+\sum_{j=1}^{n}x_j\right)\left(\sum_{j=1}^{n}x_j\right)G_g(y,x_1,\dots,x_n).
\end{multline*}
The verification is tedious but straightforward. It will be included
in a updated version of the paper \cite{LX2}.

Recall the well-known string equation
$$\langle\tau_0\prod_{i=1}^{n}\tau_{k_i}\rangle_g=\sum_{j=1}^{n}\langle\tau_{k_j-1}\prod_{i\neq
j}\tau_{k_i}\rangle_g$$ and the dilaton equation
$$\langle\tau_1\prod_{i=1}^{n}\tau_{k_i}\rangle_g=(2g-2+n)\langle\prod_{i=1}^{n}\tau_{k_i}\rangle_g.$$
Note that the string equation can be equivalently written as
$$F(x_1,\dots,x_n,0)=(\sum_{j=1}^n x_j) F(x_1,\dots,x_n).$$

\begin{proposition}Let $n\geq2$. We have the following
recursive formula of $n$-point functions.
\begin{multline*}
(2g+n-1)F_g(x_1,\dots,x_n)=\frac{\left(\sum_{j=1}^n x_j\right)^3}{12} F_{g-1}(x_1,\dots,x_n)\\
+\frac{1}{2\left(\sum_{j=1}^n
x_j\right)}\sum_{g'=0}^g\sum_{\underline{n}=I\coprod J}
\left(\sum_{i\in I} x_i\right)^2\left(\sum_{i\in J} x_i\right)^2
F_{g'}(x_I)F_{g-g'}(x_J).
\end{multline*}
\end{proposition}
\begin{proof}
From Proposition 2.2, we have
\begin{multline*}
G_g(x_1,\dots,x_n)=\sum_{r+s=g}\frac{(2r+n-3)!!}{4^s(2g+n-1)!!}P_r(x_1,\dots,x_n)\Delta(x_1,\dots,x_n)^s\\
=\frac{1}{2g+n-1}P_g(x_1,\dots,x_n)+\sum_{r+s=g-1}\frac{(2r+n-3)!!}{4^{s+1}(2g+n-1)!!}P_r(x_1,\dots,x_n)\Delta(x_1,\dots,x_n)^{s+1}\\
=\frac{1}{(2g+n-1)}P_g(x_1,\dots,x_n)+\frac{\Delta(x_1,\dots,x_n)}{4(2g+n-1)}G_{g-1}(x_1,\dots,x_n).
\end{multline*}

We define
$$H=\exp\left(\frac{\sum_{i=1}^n
x_i^3}{24}\right),\qquad H^{-1}=\exp\left(\frac{-\sum_{i=1}^n
x_i^3}{24}\right),$$
$$H_d=\frac{1}{d!}\left(\frac{\sum_{i=1}^n x_i^3}{24}\right)^d,\qquad H^{-1}_d=\frac{1}{d!}\left(\frac{-\sum_{i=1}^n x_i^3}{24}\right)^d.$$
Note that $\sum_{i=0}^d H_iH^{-1}_{d-i}=0$ if $d>0$.

Let LHS and RHS denote the left and right hand side of the recursion
in the lemma. We have
\begin{multline*}
H^{-1}\cdot RHS=\sum_{g=0}^\infty\left(\frac{1}{12}\left(\sum_{i=1}^n x_i\right)^3 G_{g-1}(x_1,\dots,x_n)+P_g(x_1,\dots,x_n)\right)\\
=\sum_{g=0}^\infty
\left((2g+n-1)G_g(x_1,\dots,x_n)+\frac{1}{12}\left(\sum_{i=1}^n
x_i^3\right) G_{g-1}(x_1,\dots,x_n)\right)
\end{multline*}

\begin{multline*}
H^{-1}\cdot LHS=\sum_{g=0}^\infty\sum_{a+b+c=g}(2a+2b+n-1)G_a(x_1,\dots,x_n)H_b H^{-1}_c\\
=\sum_{g=0}^\infty\sum_{a=0}^g(2a+n-1)G_a(x_1,\dots,x_n)\sum_{b+c=g-a}H_b
H^{-1}_c\\
+\sum_{g=0}^\infty\sum_{a+b+c=g}G_a(x_1,\dots,x_n)2b H_b H^{-1}_c\\
=\sum_{g=0}^\infty(2g+n-1)G_g(x_1,\dots,x_n)+\sum_{g=0}^\infty\frac{1}{12}\left(\sum_{i=1}^n
x_i^3\right) G_{g-1}(x_1,\dots,x_n).
\end{multline*}

\end{proof}

Only very recently, we realize that Proposition 2.3 has already been
embodied in the first KdV equation of the Witten-Kontsevich theorem.

The KdV hierarchy is the following hierarchy of differential
equations for $n\geq 1$,
$$\frac{\partial U}{\partial t_n}=\frac{\partial}{\partial t_0}R_{n+1},$$
where $R_n$ are Gelfand-Dikii differential polynomials in
$U,\partial U/\partial t_0,\partial^2 U/\partial t_0^2,\dots$,
defined recursively by
$$R_1=U,\qquad \frac{\partial R_{n+1}}{\partial t_0}=\frac{1}{2n+1}\left(\frac{\partial U}{\partial t_0}R_n+2U\frac{\partial R_n}{\partial t_0}+
\frac{1}{4}\frac{\partial^3}{\partial t_0^3}R_n\right).$$

 It is easy to see that
$$R_2=\frac{1}{2}U^2+\frac{1}{12}\frac{\partial^2U}{\partial
t_0^2},$$
$$R_3=\frac{1}{6}U^3+\frac{U}{12}\frac{\partial^3U}{\partial
t_0^3}+\frac{1}{24}(\frac{\partial U}{\partial
t_0})^2+\frac{1}{240}\frac{\partial^4U}{\partial t_0^4},$$
$$\vdots$$

The Witten-Kontsevich theorem states that the generating function
$$F(t_0, t_1, \ldots)= \sum_{g} \sum_{\bold n} \langle\prod_{i=0}^\infty \tau_{i}^{n_i}\rangle_{g} \prod_{i=0}^\infty
\frac{t_i^{n_i} }{n_i!}$$ is a $\tau$-function for the KdV
hierarchy, i.e. $\partial^2 F/\partial t_0^2$ obeys all equations in
the KdV hierarchy. The first equation in the KdV hierarchy is the
classical KdV equation
$$\frac{\partial U}{\partial t_1}=U\frac{\partial U}{\partial t_0}+\frac{1}{12}\frac{\partial^3 U}{\partial t_0^3}.$$

By the Witten-Kontsevich theorem, we have
$$\frac{\partial^3 F}{\partial t_1\partial t_0^2}=\frac{\partial F}{\partial t_0^2}\frac{\partial F}{\partial t_0^3}+
\frac{1}{12}\frac{\partial^5 F}{\partial t_0^5}.$$ Integrating each
side with respect to $t_0$ and putting
$\langle\langle\tau_{k_1}\cdots\tau_{k_n}\rangle\rangle:=\partial^nF/\partial
t_{k_1}\cdots\partial t_{k_n},$, we get
$$\langle\langle\tau_0\tau_1\rangle\rangle=\frac{1}{12}\langle\langle\tau_0^4\rangle\rangle
+\frac{1}{2}\langle\langle\tau_0^2\rangle\rangle\langle\langle\tau_0^2\rangle\rangle.$$
Then Proposition 2.3 follows by applying the dilaton equation.

\vskip 30pt
\section{The Faber intersection number conjecture}
Now we explain our approach to prove identity \eqref{eqfa3}, hence
the Faber intersection number conjecture. We establish its
relationship with $n$-point functions.

For the sake of brevity, we introduce the following notations
$$
L_g^{a,b}(y,x_1\dots,x_n) =\sum_{g'=0}^g\sum_{\underline{n}=I\coprod
J} (y+\sum_{i\in I}x_i)^a (-y+\sum_{i\in J}x_i)^b
F_{g'}(y,x_I)F_{g-g'}(-y,x_J),
$$
where $a, b\in\mathbb Z$. We regard $L_g^{a,b}(y,x_1\dots,x_n)$ as a
formal series in $\mathbb Q[x_1,\dots,x_n][[y,y^{-1}]]$ with $\deg
y<\infty$.

We now prove that the Faber intersection number conjecture can be
reduced to three statements about the coefficients of the above
functions.
\begin{proposition} We have
\begin{enumerate}
\item[i)]
$$\left[L^{0,0}_g(y,x_1\dots,x_n)\right]_{y^{2g-2}}=0;$$

\item[ii)] For $k>2g$,
$$\left[L^{2,2}_g(y,x_1\dots,x_n)\right]_{y^k}=0;$$

\item[iii)] For
$d_j\geq 1$ and $\sum_{j=1}^n d_j=g+n$,
$$\left[L^{2,2}_g(y,x_1\dots,x_n)\right]_{y^{2g}\prod_{j=1}^n
x_j^{d_j}}= \frac{(2g+n+1)!}{4^g(2g+1)!\prod_{j=1}^n(2d_j-1)!!}.$$
\end{enumerate}
\end{proposition}
In fact, Proposition 3.1 is a special case of more general results
proved in the next section. Clearly identities (i) and (ii) of the
following corollary add up to the desired identity \eqref{eqfa3}.

\begin{corollary} We have
\begin{enumerate}
\item[i)] Let $d_j\geq0$ and $\sum_{j=1}^{n}d_j=g+n-2$. Then
\begin{multline*}
\langle\prod_{j=1}^n\tau_{d_j}\tau_{2g}\rangle_g =\sum_{j=1}^{n}
\langle\tau_{d_j+2g-1}\prod_{i\neq
j}\tau_{d_i}\rangle_g-\frac{1}{2}\sum_{\underline{n}=I\coprod
J}\sum_{j=0}^{2g-2}(-1)^j\langle\tau_{j}\prod_{i\in
I}\tau_{d_i}\rangle_{g'}\langle\tau_{2g-2-j}\prod_{i\in
J}\tau_{d_i}\rangle_{g-g'};
\end{multline*}

\item[ii)] Let $d_j\geq1$ and $\sum_{j=1}^{n}(d_j-1)=g-1$. Then
$$\sum_{j=0}^{2g}(-1)^j\langle\tau_{2g-j}\tau_j\prod_{i=1}^n\tau_{d_i}\rangle_{g}=
\frac{(2g+n-1)!}{4^{g}(2g+1)!\prod_{j=1}^n(2d_j-1)!!};$$

\item[iii)] Let $k>g$, $d_j\geq0$ and $\sum_{j=1}^{n}d_j=3g+n-2k-2$. Then
$$\sum_{j=0}^{2k}(-1)^j\langle\tau_{2k-j}\tau_j\prod_{i=1}^n\tau_{d_i}\rangle_{g}=0.$$
\end{enumerate}
\end{corollary}
\begin{proof}
Since one and two-point functions in genus $0$ are
$$F_0(x)=\frac{1}{x^2},\qquad F_{0}(x,y)=\frac{1}{x+y}=\sum_{k=0}^\infty (-1)^k\frac{x^k}{y^{k+1}},$$
it is consistent to define
$$\langle\tau_{-2}\rangle_0=1,\qquad \langle\tau_{k}\tau_{-1-k}\rangle_0=(-1)^k,\ k\geq0.$$

By allowing the index to run over all integers, we have

\begin{multline*}
\frac{1}{2}\sum_{\underline{n}=I\coprod
J}\sum_{j=0}^{2g-2}(-1)^j\langle\tau_{j}\prod_{i\in
I}\tau_{d_i}\rangle_{g'}\langle\tau_{2g-2-j}\prod_{i\in
J}\tau_{d_i}\rangle_{g-g'}
+\langle\prod_{j=1}^n\tau_{d_j}\tau_{2g}\rangle_g- \sum_{j=1}^{n}
\langle\tau_{d_j+2g-1}\prod_{i\neq j}\tau_{d_i}\rangle_g\\
=\frac{1}{2}\sum_{\underline{n}=I\coprod J}\sum_{j\in\mathbb
Z}(-1)^j\langle\tau_{j}\prod_{i\in
I}\tau_{d_i}\rangle_{g'}\langle\tau_{2g-2-j}\prod_{i\in
J}\tau_{d_i}\rangle_{g-g'}\\
=\left[\sum_{g'=0}^g\sum_{\underline{n}=I\coprod J}
F_{g'}(y,x_I)F_{g-g'}(-y,x_J)\right]_{y^{2g-2}\prod_{i=1}^n
x_i^{d_i}}\\
=\left[L_g^{0,0}(y,x_1,\dots,x_n)\right]_{y^{2g-2}\prod_{i=1}^n
x_i^{d_i}}=0.
\end{multline*}

From Proposition 2.3, we have
\begin{multline*}
\frac{1}{2}\left(\sum_{j=1}^n
x_j\right)F_{g}(x_1,\dots,x_n)=\frac{\left(\sum_{j=1}^n x_j\right)^4}{24(2g+n-1)} F_{g-1}(x_1,\dots,x_n)\\
+\frac{1}{2(2g+n-1)}\left(L_{g}^{2,2}(y,x_{\underline{n}})+\sum_{g'=0}^g\sum_{\underline{n}=I\coprod
J} \left(\sum_{i\in I} x_i\right)^2\left(\sum_{i\in J} x_i\right)^2
F_{g'}(y,-y,x_I)F_{g-g'}(x_J)\right).
\end{multline*}

By Proposition 3.1(ii)-(iii), we can use Proposition 2.3 to
inductively prove
$$\sum_{j=0}^{2k}(-1)^j\langle\tau_{2k-j}\tau_j\prod_{i=1}^n\tau_{d_i}\rangle_{g}=
\left[F_{g}(y,-y,x_1,\dots,x_n)\right]_{y^{2k}}=0,\quad\text{for
}k>g$$ and we have
$$\sum_{j=0}^{2g}(-1)^j\langle\tau_{2g-j}\tau_j\tau_0\prod_{i=1}^n\tau_{d_i}\rangle_{g}=
\frac{(2g+n)!}{4^{g}(2g+1)!\prod_{j=1}^n(2d_j-1)!!},$$ which, from
the string equation and induction on the maximum index (say $d_1$)
among $\{d_i\}$, implies (by the dilaton equation, we may assume
$d_i\geq 2$)
\begin{align*}
&\sum_{j=0}^{2g}(-1)^j\langle\tau_{2g-j}\tau_j\prod_{i=1}^n\tau_{d_i}\rangle_g\\
&=\sum_{j=0}^{2g}(-1)^j\langle\tau_0\tau_{2g-j}\tau_j\tau_{d_1+1}\prod_{i=2}^n\tau_{d_i}\rangle_g-
\sum_{k=2}^n\sum_{j=0}^{2g}(-1)^j\langle\tau_{2g-j}\tau_j\tau_{d_1+1}\tau_{d_k-1}\prod_{i\neq 1,k}\tau_{d_i}\rangle_g\\
&=\frac{(2g+n)!}{4^g(2g+1)!\prod_{j=1}^n(2d_j-1)!!(2d_1+1)}-\sum_{k=2}^n\frac{(2g+n-1)!(2d_k-1)}{4^g(2g+1)!\prod_{j=1}^n(2d_j-1)!!(2d_1+1)}\\
&=\frac{(2g+n-1)!}{4^g(2g+1)!\prod_{j=1}^n(2d_j-1)!!}.
\end{align*}
\end{proof}

So in order to prove the Faber intersection number conjecture, we
only need to prove the three statements (i)-(iii) in Proposition 3.1
about $n$-point functions. Actually we will prove more general
results which are stated as main theorems, Theorems 4.4 and 4.5 in
the next section. Proposition 3.1, therefore the Faber intersection
number conjecture, is a special case of these theorems.

\vskip 30pt
\section{Proof of main theorems}

The binomial coefficients $\binom{p}{k}$, for $k\geq0, p\in\mathbb
Z$ are given by
$$
\binom{p}{k}=\begin{cases} 0, & k<0,\\
1, & k=0,\\
\frac{p(p-1)\cdots(p-k+1)}{k!}, & k\geq1 .
\end{cases}
$$

\begin{lemma}
Let $a,b\in\mathbb Z$ and $n\geq 0$. Then
$$\sum_{i=0}^n\binom{i+a}{i}\binom{n-i+b}{n-i}=\binom{n+a+b+1}{n}.$$
\end{lemma}
\begin{proof} Note that
$$\binom{p}{k}=\binom{p-1}{k}+\binom{p-1}{k-1}.$$
By denoting the left-hand side of the above equation by $A_n(a,b)$,
we have
$$A_n(a,b)=A_n(a-1,b)+A_{n-1}(a,b).$$

First we argue by induction on $n$ and $|b|$ to prove
$$A_n(0,b)=\binom{n+b+1}{n}.$$
Then we argue by induction on $n$ and $|a|$ to prove
$$A_n(a,b)=\binom{n+a+b+1}{n}.$$
\end{proof}

We now prove two lemmas that will serve as base cases for our
inductive arguments.

\begin{lemma}
Let $a,b\in\mathbb Z$ and $k\geq 2g-3+a+b$. Then
\begin{enumerate}
\item[i)] $$\left[L_g^{a,b}(y,x)\right]_{y^k}=0,$$

\item[ii)]
$$\left[L_g^{a,b}(y,x)\right]_{y^{2g-4+a+b}x^{g+1}}=
\frac{(-1)^b(2g-2+a+b)}{4^g (2g+1)!!}.$$
\end{enumerate}
\end{lemma}
\begin{proof}
Here we recall the definition of normalized $n$-point functions
$$G(x_1,\dots,x_n)=\exp\left(\frac{-\sum_{j=1}^n
x_j^3}{24}\right)\cdot F(x_1,\dots,x_n).$$ In particular, we have
$$G(x)=\frac{1}{x^2},\qquad G(x,y)=\frac{1}{x+y}\sum_{k\geq0}\frac{k!}{(2k+1)!}\left(\frac{1}{2}xy(x+y)\right)^k.$$
By definition
\begin{multline*}
\sum_{g\geq 0} L_g^{a,b}(y,x_1\dots,x_n)\\
=\exp\left(\frac{\sum_{j=1}^n
x_j^3}{24}\right)\sum_{\underline{n}=I\coprod J} (y+\sum_{i\in
I}x_i)^a (-y+\sum_{i\in J}x_i)^b G(y,x_I)G(-y,x_J),
\end{multline*}

So for statements (i) and (ii), it is not difficult to see that we
only need to prove
$$\left[y^{a-2}(-y+x)^b G_g(-y,x)+(-y)^{b-2}(y+x)^a G_g(y,x)\right]_{y^k}=0,\quad\text{for }k\geq 2g-3+a+b,$$
and
$$\left[y^{a-2}(-y+x)^b G_g(-y,x)+(-y)^{b-2}(y+x)^a G_g(y,x)\right]_{y^{2g-4+a+b} x^{g+1}}=\frac{(-1)^b(2g-2+a+b)}{4^g (2g+1)!!}.$$
Both follow easily from the explicit formula of $G(y,x)$.
\end{proof}

\begin{lemma}
Let $a,b\in\mathbb Z$ and $k\geq a+b-3$. Then
\begin{enumerate}
\item[i)] $$\left[L_0^{a,b}(y,x_1,\dots,x_n)\right]_{y^k}=0,$$

\item[ii)]
$$\left[L_0^{a,b}(y,x_1,\dots,x_n)\right]_{y^{a+b-4}\prod_{j=1}^nx_j}=
\frac{(-1)^b(a+b+n-3)!}{(a+b-3)!}.$$
\end{enumerate}
\end{lemma}
\begin{proof} Since
$$F_0(x_1,\dots,x_n)=(x_1+\cdots+x_n)^{n-3},$$
we have by definition
$$
L_0^{a,b}(y,x_1\dots,x_n) =\sum_{\underline{n}=I\coprod J}
(y+\sum_{i\in I}x_i)^{|I|-2+a} (-y+\sum_{i\in J}x_i)^{|J|-2+b}.
$$
For any monomial $y^k\prod_{j=1}^n x_j^{d_j}$ in
$L_0^{a,b}(y,x_1\dots,x_n)$, if $k\geq a+b-3$, then there must be
some $d_j=0$. We may assume $d_n=0$, then
\begin{multline*}
L_0^{a,b}(y,x_1\dots,x_{n-1},0)\\
=\sum_{\{1,\dots,n-1\}=I\coprod J}\left((y+\sum_{i\in
I}x_i)^{|I|-1+a}(-y+\sum_{i\in J}x_i)^{|J|-2+b}\right.\\
\left.+(y+\sum_{i\in I}x_i)^{|I|-2+a}(-y+\sum_{i\in J}x_i)^{|J|-1+b}\right)\\
=\left(\sum_{j=1}^{n-1} x_j\right)\sum_{\{1,\dots,n-1\}=I\coprod
J}(x_1+\sum_{i\in I}x_i)^{|I|-2+a}(-x_1+\sum_{i\in
J}x_i)^{|J|-2+b}\\
=\left(\sum_{j=1}^{n-1} x_j\right)L_0^{a,b}(y,x_1\dots,x_{n-1}).
\end{multline*}

So (i) follows by induction on $n$. By applying Lemma 4.1 we have

\begin{align*}
&\left[L_0^{a,b}(y,x_1,\dots,x_n)\right]_{y^{a+b-4}\prod_{j=1}^nx_j}\\
&=(-1)^b\sum_{|I|=0}^n\binom{|I|-2+a}{|I|}|I|!\binom{|j|-2+b}{|J|}|J|!\binom{n}{|I|}\\
&=(-1)^b n!\sum_{i=0}^n\binom{i-2+a}{i}\binom{n-i-2+b}{n-i}\\
&=(-1)^b n! \binom{a+b+n-3}{n}\\ &= \frac{(-1)^b
(a+b+n-3)!}{(a+b-3)!}.
\end{align*}

So we have proved (ii).
\end{proof}

\begin{theorem}
Let $a,b\in\mathbb Z$ and $k\geq 2g-3+a+b$. Then
$$\left[L_g^{a,b}(y,x_1\dots,x_n)\right]_{y^k}=0.$$
\end{theorem}
\begin{proof}  We will argue by induction on $g$ and
$n$, since the theorem holds for $g=0$ or $n=1$ as proved in the
above lemmas. We have
\begin{multline*}
(2g+n)L_g^{a,b}(y,x_1\dots,x_n)\\
=\sum_{g'=0}^g\sum_{\underline{n}=I\coprod J} (y+\sum_{i\in I}x_i)^a
(-y+\sum_{i\in J}x_i)^b
(2g'+|I|)F_{g'}(y,x_I)F_{g-g'}(-y,x_J)\\
+\sum_{g'=0}^g\sum_{\underline{n}=I\coprod J} (y+\sum_{i\in I}x_i)^a
(-y+\sum_{i\in J}x_i)^b F_{g'}(y,x_I)(2g-2g'+|J|)F_{g-g'}(-y,x_J).
\end{multline*}

Substituting $F_{g'}(y,x_I)$ by Propostion 2.3,
\begin{multline*}
\left[\sum_{g'=0}^g\sum_{\underline{n}=I\coprod J} (y+\sum_{i\in
I}x_i)^a (-y+\sum_{i\in J}x_i)^b
(2g'+|I|)F_{g'}(y,x_I)F_{g-g'}(-y,x_J)\right]_{y^k}\\
=\frac{1}{12}\left[L^{a+3,b}_{g-1}(y,x_1,\dots,x_n)\right]_{y^k}
+\left[\sum_{g'=0}^g\sum_{s\geq0}\binom{a-1}{s}\sum_{\underline{n}=I\coprod
J}F_{g'}(x_I)(\sum_{i\in
I}x_i)^{s+2}L^{a+1-s,b}_{g-g'}(y,x_J)\right]_{y^k}.
\end{multline*}

Note that in the last term of the above equation, $|J|<n$. So by
induction, for $k\geq 2g-3+a+b$, the sums vanish except for $g'=0$
and $s=0$, namely the term
$$\left[\sum_{\underline{n}=I\coprod
J}(\sum_{i\in I}x_i)^{|I|-1} L^{a+1,b}_{g}(y,x_J)\right]_{y^k}.$$

Let $d_j\geq 1$ for $1\leq j\leq n$. By induction, it is not
difficult to see from the above that
\begin{multline*}
(2g+n)\left[L_g^{a,b}(y,x_1\dots,x_n)\right]_{y^k\prod_{j=1}^n
x_j^{d_j}}
\\=\frac{1}{12}\left[L^{a+3,b}_{g-1}(y,x_1,\dots,x_n)+
L^{a,b+3}_{g-1}(y,x_1,\dots,x_n)\right]_{y^k\prod_{j=1}^nx_j^{d_j}}.
\end{multline*}
By induction, we have
\begin{align*}
0&=\left(\sum_{j=1}^n x_j\right)
\left[L^{a+1,b+1}_{g-1}(y,x_1,\dots,x_n)\right]_{y^k}\quad
\text{for }k\geq 2g-3+a+b\\
&=\left[L^{a+2,b+1}_{g-1}(y,x_1,\dots,x_n)+L^{a+1,b+2}_{g-1}(y,x_1,\dots,x_n)\right]_{y^k}
\end{align*}
and
\begin{align*}
0=&\left(\sum_{j=1}^{n} x_j\right)^3
\left[L^{a,b}_{g-1}(y,x_1,\dots,x_n)\right]_{y^k}\quad
\text{for }k\geq 2g-5+a+b\\
=&\left[L^{a+3,b}_{g-1}(y,x_1,\dots,x_n)+L^{a,b+3}_{g-1}(y,x_1,\dots,x_n)\right]_{y^k}\\
&+3\left[L^{a+2,b+1}_{g-1}(y,x_1,\dots,x_n)+
L^{a+1,b+2}_{g-1}(y,x_1,\dots,x_n)\right]_{y^k}\\
=&\left[L^{a+3,b}_{g-1}(y,x_1,\dots,x_n)+L^{a,b+3}_{g-1}(y,x_1,\dots,x_n)\right]_{y^k}.
\end{align*}
So we have proved that
$$\left[L_g^{a,b}(y,x_1\dots,x_n)\right]_{y^k\prod_{j=1}^n x_j^{d_j}}=0, \quad \text{for }d_j\geq 1.$$
If some $d_j$ is zero, the above identity still holds by applying
the string equation
$$L_g^{a,b}(y,x_1\dots,x_n,0)=\left(\sum_{j=1}^n x_j\right)L_g^{a,b}(y,x_1\dots,x_n).$$
So we proved the theorem.
\end{proof}

\begin{theorem}
Let $a,b\in\mathbb Z$, $d_j\geq 1$ and $\sum_{j}d_j=g+n$. Then
$$\left[L_g^{a,b}(y,x_1\dots,x_n)\right]_{y^{2g-4+a+b}\prod_{j=1}^nx_j^{d_j}}=
\frac{(-1)^b(2g-3+n+a+b)!}{4^g(2g-3+a+b)!\prod_{j=1}^n(2d_j-1)!!}.$$
\end{theorem}
\begin{proof} By the dilaton equation, we may assume $d_j\geq 2$. As in the proof of the above theorem, we have
\begin{align*}
&(2g+n)\left[L_g^{a,b}(y,x_1\dots,x_n)\right]_{y^{2g-4+a+b}\prod_{j=1}^n x_j^{d_j}}\\
&=\frac{1}{12}\left[L^{a+3,b}_{g-1}(y,x_{\underline{n}})+
L^{a,b+3}_{g-1}(y,x_{\underline{n}})\right]_{y^{2g-4+a+b}\prod_{j=1}^n x_j^{d_j}}\\
&=-\frac{1}{4}\left[L^{a+2,b+1}_{g-1}(y,x_{\underline{n}})+
L^{a+1,b+2}_{g-1}(y,x_{\underline{n}})\right]_{y^{2g-4+a+b}\prod_{j=1}^n x_j^{d_j}}\\
&=-\frac{1}{4}\left[\left(\sum_{i=1}^n
x_i\right)L^{a+1,b+1}_{g-1}(y,x_{\underline{n}})\right]_{y^{2g-4+a+b}\prod_{j=1}^n x_j^{d_j}}\\
&=-\frac{1}{4}\sum_{j=1}^n\left[L^{a+1,b+1}_{g-1}(y,x_{\underline{n}})\right]_{y^{2g-4+a+b}x_j^{d_j-1}\prod_{i\neq
j} x_i^{d_i}}\\
&=\frac{(-1)^b(2g-3+n+a+b)!}{4^g(2g-3+a+b)!\prod_{j=1}^n(2d_j-1)!!}\sum_{j=1}^n(2d_j-1)\\
&=(2g+n)\frac{(-1)^b(2g-3+n+a+b)!}{4^g(2g-3+a+b)!\prod_{j=1}^n(2d_j-1)!!}.
\end{align*}

So we have proved the theorem.
\end{proof}

All the three statements in Proposition 3.1 are particular cases of
Theorems 2.4 and 2.5. We thus conclude the proof of the Faber
intersection number conjecture.

The following corollaries were stated as conjectures in our previous
paper \cite{LX1}.
\begin{corollary}
Let $d_j\geq1$ and $\sum_{j=1}^{n}(d_j-1)=g$. Then
\begin{multline*}
\frac{(2g-3+n)!}{2^{2g+1}(2g-3)!\prod_{j=1}^n(2d_j-1)!!}
=\langle\tau_{2g-2}\prod_{j=1}^n\tau_{d_j}\rangle_g-\sum_{j=1}^{n}
\langle\tau_{d_j+2g-3}\prod_{i\neq j}\tau_{d_i}\rangle_g\\
+\frac{1}{2}\sum_{\underline{n}=I\coprod
J}\sum_{j=0}^{2g-4}(-1)^j\langle\tau_{j}\prod_{i\in
I}\tau_{d_i}\rangle_{g'}\langle\tau_{2g-4-j}\prod_{i\in
J}\tau_{d_i}\rangle_{g-g'}.
\end{multline*}
\end{corollary}
\begin{proof}
Since the right hand side is just
$$\frac{1}{2}\left[L_g^{0,0}(y,x_1\dots,x_n)\right]_{y^{2g-4}\prod_{j=1}^nx_j^{d_j}},$$
the result follows from Theorem 4.5.
\end{proof}

\begin{corollary}
Let $g\geq2$, $d_j\geq1$ and $\sum_{j=1}^{n}(d_j-1)=g$. Then
\begin{multline*}
-\frac{(2g-2)!}{|B_{2g-2}|}\int_{\overline{\sM}_{g,n}}\psi_1^{d_1}\cdots\psi_n^{d_n}{\rm ch}_{2g-3}(\mathbb E)\\
=\frac{2g-2}{|B_{2g-2}|}\left(\int_{\overline{\sM}_{g,n}}\psi_1^{d_1}\cdots\psi_n^{d_n}\lambda_{g-1}\lambda_{g-2}-3\int_{\overline{\sM}_{g,n}}\psi_1^{d_1}\cdots\psi_n^{d_n}\lambda_{g-3}\lambda_{g}\right)\nonumber\\
=\frac{1}{2}\sum_{j=0}^{2g-4}(-1)^j\langle\tau_{2g-4-j}\tau_j\tau_{d_1}\cdots\tau_{d_n}\rangle_{g-1}+\frac{(2g-3+n)!}{2^{2g+1}(2g-3)!}\cdot\frac{1}{\prod_{j=1}^n(2d_j-1)!!}.
\end{multline*}
\end{corollary}
\begin{proof} We apply Mumford's formulae \cite{Mu}
$$(2g-3)!\cdot{\rm ch}_{2g-3}(\mathbb
E)=(-1)^{g-1}(3\lambda_{g-3}\lambda_g-\lambda_{g-1}\lambda_{g-2}),$$
$$
{\rm ch}_{2g-3}(\mathbb E)= \frac{ B_{2g-2}}{(2g-2)!}\left[
\kappa_{2g-3}-\sum_{i=1}^{n}\psi_i^{2g-3}+\frac12\sum_{\xi\in\Delta}{l_{\xi}}_*
\left(\sum_{i=0}^{2g-4}\psi_{n+1}^i(-\psi_{n+2})^{2g-4-i}\right)\right].
$$
So the identity follows from Corollary 4.6.
\end{proof}

Both Theorems 4.4 and 4.5 can be extended without difficulty.

Let us use the notation
$$
L_g(y,z_{\underline{a}},w_{\underline{b}},x_{\underline{n}})
=\sum_{g'=0}^g\sum_{\underline{n}=I\coprod J}
F_{g'}(y,z_1,\dots,z_a,x_I)F_{g-g'}(-y,w_1,\dots,w_b,x_J).
$$

\begin{theorem} Let $a\geq0$, $b\geq0$, $n\geq1$. We have
\begin{enumerate}
\item[i)] For $k\geq 2g-3+a+b$,
$$\left[L_g(y,z_{\underline{a}},w_{\underline{b}},x_{\underline{n}})\right]_{y^k}=0.$$

\item[ii)] For $r_j\geq0$, $s_j\geq0$, $d_j\geq 1$ and $\sum r_j +\sum
s_j+\sum d_j=g+n$,
\begin{multline*}
\left[L_g(y,z_{\underline{a}},w_{\underline{b}},x_{\underline{n}})\right]_
{y^{2g-4+a+b}\prod_{j=1}^a z_j^{r_j}\prod_{j=1}^b
w_j^{s_j}\prod_{j=1}^n
x_j^{d_j}}\\
=\frac{1}{\prod_{j=1}^a (2r_j+1)!!\prod_{j=1}^b
(2s_j+1)!!}\cdot\frac{(-1)^b(2g-3+n+a+b)!}{4^g(2g-3+a+b)!\prod_{j=1}^n(2d_j-1)!!}.
\end{multline*}

\item[iii)] For $r_j\geq0$, $s_j\geq0$, $d_j\geq 1$, $\sum r_j +\sum
s_j+\sum d_j=g+n+1$ and $u\triangleq \#\{r_j=0\}$,
$v\triangleq\#\{s_j=0\}$, $w\triangleq\#\{d_j=1\}$,
\begin{multline*}
\left[L_g(y,z_{\underline{a}},w_{\underline{b}},x_{\underline{n}})\right]_
{y^{2g-5+a+b}\prod_{j=1}^a z_j^{r_j}\prod_{j=1}^b
w_j^{s_j}\prod_{j=1}^n
x_j^{d_j}}\\
=\frac{C}{\prod_{j=1}^a (2r_j+1)!!\prod_{j=1}^b
(2s_j+1)!!}\cdot\frac{(-1)^b(2g-3+n+a+b)!}{4^g(2g-4+a+b)!\prod_{j=1}^n(2d_j-1)!!},
\end{multline*}
where the constant $C$ is given by $$C\triangleq\sum_{j=1}^a
r_j-\sum_{j=1}^b
s_j+\frac{a-b}{2}+\frac{(5-u)u-(5-v)v}{2(2g+n+a+b-3-w)}.$$
\end{enumerate}
\end{theorem}

\begin{proof} When $g=0$, the proof is an easy verification.
Let $p,q\in\mathbb Z$.
\begin{multline*}
L_g^{p,q}(y,z_{\underline{a}},w_{\underline{b}},x_{\underline{n}})
=\sum_{g'=0}^g\sum_{\underline{n}=I\coprod J}(y+\sum_{i=1}^a z_i+\sum_{i\in I}x_i)^p(-y+\sum_{i=1}^b w_i+\sum_{i\in J}x_i)^q\\
\times F_{g'}(y,z_1,\dots,z_a,x_I)F_{g-g'}(-y,w_1,\dots,w_b,x_J).
\end{multline*}

Exactly the same argument of Theorem 4.4 will prove that for $k\geq
2g-3+p+q+a+b$,
$$[L_g^{p,q}(y,z_{\underline{a}},w_{\underline{b}},x_{\underline{n}})]_{y^k}=0.$$

Statements (ii) and (iii) can also be proved similarly as Theorem
4.5.
\end{proof}

Theorem 4.8 proves all conjectures in Section 3 of \cite{LX1}. We
may write down the coefficients of
$L_g(y,z_{\underline{a}},w_{\underline{b}},x_{\underline{n}})$
explicitly to get a lot of interesting identities of intersection
numbers. For example, when $a=1, b=0$,
\begin{multline*}
\left[L_g(y,z,x_{\underline{n}})\right]_{y^k z^r \prod_{j=1}^n
x_j^{d_j}}=\sum_{\underline{n}=I\coprod
J}\sum_{j=0}^{k}(-1)^j\langle\tau_j\prod_{i\in
I}\tau_{d_i}\rangle_{g'}\langle\tau_{k-j}\tau_r\prod_{i\in
J}\tau_{d_i}\rangle_{g-g'}\\
+\langle\tau_{k+2}\tau_r\prod_{j=1}^n\tau_{d_j}\rangle_g-(-1)^k\langle\tau_{k+r+1}\prod_{j=1}^n\tau_{d_j}\rangle_g
-\sum_{j=1}^n\langle\tau_r\tau_{d_j+k+1}\prod_{i\neq
j}\tau_{d_i}\rangle_g.
\end{multline*}

When $a=b=1$,
\begin{multline*}
\left[L_g(y,z,w,x_{\underline{n}})\right]_{y^k z^r w^s\prod_{j=1}^n
x_j^{d_j}}=\sum_{\underline{n}=I\coprod
J}\sum_{j=0}^{k}(-1)^j\langle\tau_j\tau_s\prod_{i\in
I}\tau_{d_i}\rangle_{g'}\langle\tau_{k-j}\tau_r\prod_{i\in
J}\tau_{d_i}\rangle_{g-g'}\\
-\langle\tau_{k+s+1}\tau_r\prod_{j=1}^n\tau_{d_j}\rangle_g-(-1)^k\langle\tau_{k+r+1}\tau_s\prod_{j=1}^n\tau_{d_j}\rangle_g.
\end{multline*}

\vskip 30pt
\section{Gromov-Witten invariants}

We will generalize vanishing identities in previous sections to
Gromov-Witten invariants.

Let $X$ be a smooth projective variety and $\overline{\mathcal
M}_{g,n}(X,\beta)$ denote the moduli stack of stable maps of genus
$g$ and degree $\beta\in H_2(X, \mathbb Z)$ with $n$ marked points.
There are several canonical morphisms:
\begin{enumerate}
\item[i)] Let ${\rm ev}: \overline{\mathcal
M}_{g,n}(X,\beta)\rightarrow X^n$ be the evaluation maps at the
marked points:
$$
{\rm ev} : {(f:C\to X, x_1,\dots,x_n) } \mapsto \bigl(
f(x_1),\dots,f(x_n) \bigr) \in X^n .
$$

\item[ii)] Let
$\pi : \overline{\mathcal M}_{g,n+1}(X,\beta) \rightarrow
\overline{\mathcal M}_{g,n}(X,\beta)$ be the map of forgetting the
last marked point $x_{n+1}$ and stabilizing the resulting curve.
\end{enumerate}

The forgetful morphism $\pi$ has $n$ canonical sections
$$
\sigma_i : \overline{\mathcal M}_{g,n}(X,\beta)\rightarrow
\overline{\mathcal M}_{g,n+1}(X,\beta) ,
$$
corresponding to the $n$ marked points. Let
$$
\omega=\omega_{\overline{\mathcal
M}_{g,n+1}(V,\beta)/\overline{\mathcal M}_{g,n}(X,\beta)}
$$
be the relative dualizing sheaf and $\Psi_i$ the cohomology class
$c_1(\sigma_i^*\omega)$.

If $\gamma_1,\dots,\gamma_n\in H^*(X,\mathbb Q)$, the Gromov-Witten
invariants are defined by
$$
\langle\tau_{d_1}(\gamma_1) \dots
\tau_{d_n}(\gamma_n)\rangle^V_{g,\beta} = \int_{[\overline{\mathcal
M}_{g,n}(V,\beta)]^{\rm virt}} \Psi_1^{d_1} \cdots \Psi_n^{d_n} \cup
{\rm ev}^*(\gamma_1\boxtimes \cdots \boxtimes \gamma_n).
$$

Given a basis $\{T_a\}$ for $H^*(X,\mathbb Q)$, we may use
$g_{ab}=\int_X T_a\cup T_b$ and its inverse $g^{ab}$ to lower and
raise indices. We denote by $T^a=g^{ab}T_b$ and apply the Einstein
summation convention.

The genus $g$ Gromov-Witten potential of $X$ is defined by
$$\langle\langle\tau_{d_1}(\gamma_1)\cdots\tau_{d_n}(\gamma_n)\tau\rangle\rangle_g=
\sum_{\beta}\left\langle\tau_{d_1}(\gamma_1)\cdots\tau_{d_n}(\gamma_n)\exp\left(\sum_{m,a}t_m^a\tau_m(T_a)\right)\right\rangle^X_{g,\beta}q^\beta.$$
Very readable expositions of Gromov-Witten invariants can be found
in \cite{Ge, Va}.

We adopt Gathmann's convention \cite{Ga} in the following which will
simplify the notation, namely we define
$$\langle\tau_{-2}(pt)\rangle_{0,0}^X=1,$$
$$\langle\tau_{m}(\gamma_1)\tau_{-1-m}(\gamma_2)\rangle_{0,0}^X=(-1)^{\max(m,-1-m)}\int_X\gamma_1\cdot\gamma_2,
\quad m\in\mathbb Z.$$ All other Gromov-Witten invariants that
contain a negative power of a cotangent line are defined to be zero.

Motivated by our previous results, we conjecture the following
relations for Gromov-Witten invariants, which we have checked in
various cases. We deem they are interesting constraints on
Gromov-Witten invariants.

\begin{conjecture}
Let $x_i, y_i\in H^*(X)$ and $k\geq 2g-3+r+s$. Then
$$\sum_{g'=0}^g\sum_{j\in\mathbb
Z}(-1)^j\langle\langle\tau_j(T_a)\prod_{i=1}^r\tau_{p_i}(x_i)\rangle\rangle_{g'}
\langle\langle\tau_{k-j}(T^a)\prod_{i=1}^s\tau_{q_i}(y_i)\rangle\rangle_{g-g'}=0.$$
Note that $j$ runs over all integers.
\end{conjecture}

 Conjecture 5.1 is a
direct generalization of Theorem 4.8(i) in the point case. For
example, when $r=s=0$, Conjecture 5.1 becomes
$$\langle\langle\tau_{2k}(1)\rangle\rangle_{g}-\sum_{m,a}t_m^a\langle\langle\tau_{m+2k-1}(T_a)\rangle\rangle_{g}+\frac{1}{2}
\sum_{g'=0}^g\sum_{j=0}^{2k-2}(-1)^j\langle\langle\tau_{j}(T_a)\rangle\rangle_{g'}\langle\langle\tau_{2k-2-j}(T^a)\rangle\rangle_{g-g'}=0$$
for $k\geq g$.

\begin{conjecture}
Let $k>g$. Then
\begin{equation} \label{eqgw1}
\sum_{j=0}^{2k}(-1)^j\langle\langle\tau_{j}(T_a)\tau_{2k-j}(T^a)\rangle\rangle_g^X=0.
\end{equation}
We also have
\begin{equation} \label{eqgw2}
\frac12\sum_{j=0}^{2g-2}(-1)^j\langle\langle\tau_{j}(T_a)\tau_{2g-2-j}(T^a)\rangle\rangle_{g-1}=\frac{(2g)!}{B_{2g}}\langle\langle{\rm
ch}_{2g-1}(\mathbb E)\rangle\rangle_g.
\end{equation}
\end{conjecture}

Similar vanishing conjectures 5.1 and 5.2 can also be made about
Witten's r-spin intersection numbers \cite{Wi2}. Thus these
vanishing identities should be regarded as some universal
topological recursion relations (TRR) valid in all genera.

Note that by the Chern character formula of Faber and Pandharipande
\cite{FaPa} and the fact $\rm ch_{k}(\mathbb E)=0$, $k>2g$, we have
the equivalence\bigskip

\centerline{Conjecture 5.1 ($r=s=0$) $\Longleftrightarrow$
identities \eqref{eqgw1} and \eqref{eqgw2}} \bigskip

Recently, X. Liu and R. Pandharipande \cite{LiP} give a proof of the
above Conjectures 5.1 and 5.2. Their proof uses virtual localization
to get topological recursion relations in the tautological ring of
moduli spaces of curves, which are translated into universal
equations for Gromov-Witten invariants by the splitting axiom and
cotangent line comparison equations.

Earlier, X. Liu \cite{Liu} proves the case $r=s=0$ of Conjecture 5.1
and Conjecture 5.2 \eqref{eqgw1} both for $g\leq 2$ using
topological recursion relations in low genus, which is tour de
force, since the number of terms in TRR increase very rapidly with
$g$. For example, Getzler's TRR in $g=2$ contains $15$ terms.

$$ \ \ \ \ $$

\end{document}